\documentclass[11pt]{article}
\usepackage[letterpaper,left=1in,right=1in,top=1in,bottom=1in]{geometry}
\usepackage[mathscr]{eucal}
\usepackage{epsfig,epsf,psfrag}
\usepackage{amssymb,amsfonts,latexsym}
\usepackage{amsmath}
\usepackage{graphicx}
\usepackage{epstopdf}
\usepackage{bm,xcolor,url}
\usepackage{fixltx2e}
\usepackage{array}
\usepackage{verbatim}
\usepackage[noend]{algpseudocode}
\usepackage{cite}
\usepackage{algorithm}
\usepackage{verbatim}
\usepackage{textcomp}
\usepackage{mathrsfs}
\usepackage[referable]{threeparttablex}
\usepackage{subfig}
\usepackage{amsthm}
\usepackage{enumerate}
\usepackage{footnote}
\usepackage{enumitem}
\usepackage{xr}
\usepackage{mathtools}
\catcode`~=11 \def\UrlSpecials{\do\~{\kern -.15em\lower .7ex\hbox{~}\kern .04em}} \catcode`~=13 

\allowdisplaybreaks[3]

\newcommand{\tnorm}[1]{{\left\vert\kern-0.25ex\left\vert\kern-0.25ex\left\vert #1 
    \right\vert\kern-0.25ex\right\vert\kern-0.25ex\right\vert}}
\newcommand{\tnormt}[1]{{\vert\kern-0.25ex\vert\kern-0.25ex\vert #1 
    \vert\kern-0.25ex\vert\kern-0.25ex\vert}}

\newcommand{\normt}[1]{\Vert#1\Vert}

\newcommand{\nn}{\nonumber}

\newcommand{\dom}{\mathsf{dom}\,}

\newcommand{\ri}{\mathsf{ri}\,}
\newcommand{\where}{\mathrm{where}}

\newcommand{\andd}{\mathrm{and}}

\newcommand{\diam}{\mathsf{diam}}
\newcommand{\dist}{\mathsf{dist}}

\newcommand{\bartheta}{\bar{\theta}}

\newcommand{\calC}{\mathcal{C}}

\newcommand{\calU}{\mathcal{U}}

\newcommand{\calX}{\mathcal{X}}

\newcommand{\rmd}{\mathrm{d}}

\newcommand{\bbR}{\mathbb{R}}

\newcommand{\barbbR}{\bar{\bbR}}

\DeclareMathAlphabet{\mathbsf}{OT1}{cmss}{bx}{n}

\newcommand{\tilG}{\widetilde{G}}

\newcommand{\uG}{\underline{G}}

\newcommand{\ceil}[1]{\lceil{#1}\rceil}

\newcommand{\ip}[2]{\left\langle{#1},{#2}\right\rangle}
\newcommand{\ipt}[2]{\langle{#1},{#2}\rangle}

\newcommand{\lea}{\stackrel{\rm(a)}{\le}}

\newcommand{\lb}{\stackrel{\rm(b)}{<}}

\DeclareMathOperator*{\argmin}{arg\,min}

\DeclareMathOperator{\st}{s.t.}

\newtheorem{theorem}{Theorem} 
\newtheorem*{theorem*}{Theorem}
\newtheorem{lemma}{Lemma}

\newtheorem{corollary}{Corollary}
 
\newtheorem*{assump*}{Assumption}

\theoremstyle{definition}
\newtheorem{definition}{Definition}

\theoremstyle{remark}
\newtheorem{remark}{Remark}

\newcommand{\qednew}{\nobreak \ifvmode \relax \else
      \ifdim\lastskip<1.5em \hskip-\lastskip
      \hskip1.5em plus0em minus0.5em \fi \nobreak
      \vrule height0.75em width0.5em depth0.25em\fi}

\title{On the Convergence Rate of the Duality Gap in the Generalized Frank-Wolfe Method Under the Weak Growth Condition}
\usepackage[colorlinks=true,breaklinks=true,bookmarks=true,urlcolor=blue,
     citecolor=blue,linkcolor=blue,bookmarksopen=false,draft=false]{hyperref}
\author{Renbo Zhao
\thanks{Tippie College of Business, University of Iowa, Iowa City, IA 52242 ({mailto:  renbo-zhao@uiowa.edu}).}
}

\usepackage{algorithm}
\newcommand{\gap}{{\mathsf{gap}}}
\newcommand{\wg}{{\mathsf{wgap}}}
\newcommand{\subopt}{{\mathsf{subopt}}}
\numberwithin{equation}{section}
\numberwithin{definition}{section}
\numberwithin{prop}{section}
\numberwithin{lemma}{section}
\numberwithin{theorem}{section}
\numberwithin{remark}{section}
\numberwithin{corollary}{section}
\newcommand{\Pena}{{Pe{\~n}a}\;}

\begin{document}

\maketitle

\begin{abstract}
We analyze the convergence rate of certain sequence of the duality gaps in the generalized Frank-Wolfe method, under certain regime of the weak growth condition proposed in \Pena~\cite{Pena_23}. Our analysis leverages a recursion lemma that may be of independent interest.
\end{abstract}

\section{Introduction} \label{sec:intro}

We consider the generalized Frank-Wolfe (FW) method for the solving the following convex composite minimization problem
\begin{equation}\label{eq:P}
P^*:= {\min}_{x\in \bbR^n}\, \{P(x):=f(x)+ \Psi(x)\} \tag{P}
\end{equation}
where $f:\bbR^n\to \barbbR:=\bbR\cup\{+\infty\}$ and $\Psi:\bbR^n\to \barbbR$ are proper, closed and convex (p.c.c.) functions. We assume that 
\begin{enumerate}[label=(A\arabic*)]
\item \label{assump:f_diff} $f$ is differentiable on an open set $\calU\subseteq \bbR^n$ that contains $\dom\Psi:=\{x\in\bbR^n:\Psi(x)<+\infty\}$.
\item \label{assump:LMO} For any $g\in \bbR^n$, the following optimization problem 
\begin{equation}
 {\min}_{x\in \bbR^n}\,\{\ip{g}{x}+\Psi(x)\}, \tag{LM} \label{eq:LMO}
\end{equation}
admits an easily computable optimal solution. 
\end{enumerate}
In the literature, the oracle that outputs an solution of~\eqref{eq:LMO} is often called  
the {\em (generalized) linear minimization oracle}.

 The traditional format of the FW method (e.g.,~\cite{Frank_56,Jaggi_13,Freund_16}) solves~\eqref{eq:P}  where $\Psi$ is the indicator function of some (nonempty) compact convex set, 
 and later works (e.g.,~\cite{Bach_15,Nest_18,Ghad_19}) generalize this method by allowing $\Psi$ to be any p.c.c.\ function such that both Assumptions~\ref{assump:f_diff} and~\ref{assump:LMO} hold. 
 The convergence rate of the generalized FW method (shown in Algorithm~\ref{algo:CG}) has been studied under various settings --- see e.g.,~\cite{Frank_56,Levitin_66,Jaggi_13,Bach_15,Garber_15,Freund_16,Nest_18,Xu_18,Ghad_19,Kerd_21}.
Some of the most attractive features of this method are its affine invariance, norm-independence, and avoidance on projection mappings. 

\begin{algorithm}[t]
\caption{The Generalized FW Method}\label{algo:CG}
\begin{algorithmic}[1]
	\State {\bf Input:}  $x_{0}\in \dom\Psi$, $\varepsilon>0$
	\For{$k=0,1,2,\dots$}
		\State Compute $g_k:=\nabla f(x_k)$, $s_{k} \in \argmin_{x\in \bbR^n}\{\ip{g_k}{x} + \Psi(x)\}$ and  $\theta_k \in [0,1]$
		\State Compute $G_k:=\ip{g_k}{x_k-s_k} + \Psi(x_k) - \Psi(s_k)$
		\If{$G_k\le \varepsilon$}
		\State {\bf Stop}
		\Else 
		\State Compute $x_{k+1} := (1-\theta_k) x_k + \theta_k s_k$  
		\EndIf
	\EndFor
\end{algorithmic}
\end{algorithm}

  In the recent seminal work~\cite{Pena_23}, Pe{\~n}a proposed two general growth conditions, namely the ``strong growth condition'' and ``weak growth condition'', that subsume and extend many of the settings studied previously. 
   More importantly, these two conditions are  affine-invariant and norm-independent, 
   and based on these two conditions, \Pena developed convergence-rate analyses of the generalized FW method that also enjoy the same desirable properties. Specifically, 
  under the strong growth condition, the convergence rates of both  the sequence of (primal)  sub-optimality gaps and certain sequence of  duality gaps were obtained. In contrast, under the weak growth condition, only the convergence rate of the sequence of  sub-optimality gaps was obtained  (and perhaps,  this explains the name ``weak''). In this paper, we shall derive the convergence rate of certain sequence of duality gaps under certain regime of the $(q,r)$-weak growth condition (which we detail in Section~\ref{sec:weak_growth}). Indeed,  when the objective function in the dual problem of~\eqref{eq:P} can be easily evaluated, any form of the sequence of duality gaps can be easily computed. Since by weak duality, any duality gap upper bounds the sub-optimality gap, the duality gap 
  plays an important role in algorithmic termination. As such, obtaining the convergence rate of (certain sequence of) the duality gaps is an important aspect in the analysis of the generalized FW method.
  
  To provide an overview of our results, let us note that at each iteration $k\ge 0$ in Algorithm~\ref{algo:CG}, the quantity $G_k$ is precisely the duality gap between $x_k$ and $g_k$ (cf.~\eqref{eq:wg_duality_gap_k} in Section~\ref{sec:prelim}), and Algorithm~\ref{algo:CG} is terminated once $G_k$ becomes sufficiently small (i.e., $G_k\le \varepsilon$). Let $K$ be the index of the iteration immediately before Algorithm~\ref{algo:CG} stops, and our results essentially show that under proper choices of $\theta_k$ and the regime that $r+q>2$, we have 
  \begin{equation}
G_k^{\rm best} = O\big(k^{-\frac{q-1}{1-r}}\big), \quad\where\quad G_k^{\rm best}:={\min}_{i=0,\ldots,k}\,G_i, \quad \forall\,0\le k\le K. \label{eq:def_G_k_best}
\end{equation}
(Indeed, $\{G_k^{\rm best}\}_{0\le k \le K}$  
denotes the sequence of ``best'' duality gaps, and can be readily obtained from Algorithm~\ref{algo:CG} without additional function evaluations.) 
It turns out that in our analysis, we leverage a recursion lemma (cf.~Lemma~\ref{lem:rate_G_k}) that may be of independent interest. 
In addition, when specializing our results to the specific scenarios studied in~\cite{Garber_15,Xu_18,Kerd_21},  we show that our results also provide convergence rates of $\{G_k^{\rm best}\}_{0\le k \le K}$  to these works 
(cf.~Corollary~\ref{cor:exact_LS}).  
  

\section{Preliminaries}\label{sec:prelim}

Let $f^*:\bbR^n\to\barbbR$ denote the Fenchel conjugate of $f$, namely 
\begin{equation}
f^*(u) := {\sup}_{x\in \bbR^n}\, \ipt{u}{x} - f(x). 
\end{equation}
Then the Fenchel dual of~\eqref{eq:P} reads
\begin{equation}\label{eq:D}
-d^*:= -{\min}_{u\in \bbR^n}\, \{d(u):=\Psi^*(-u)+ f^*(u)\}.  \tag{D} 
\end{equation}
Note that Assumption~\ref{assump:f_diff} 
amounts to assuming that $\partial \Psi^*(-g)\ne \emptyset$ for all $g\in\bbR^n$, which in turn amounts to assuming that $\dom \Psi^* = \bbR^n$. 
This, together with Assumption~\ref{assump:LMO}, indicates that both $\ri \dom f \cap\ri\dom \Psi\ne \emptyset$ and $(-\ri \dom f^*) \cap\ri\dom \Psi^*\ne \emptyset$, and hence by the Fenchel duality theorem (cf.~\cite[Theorem~31.1]{Rock_70}), we know that $-\infty < P^* = -d^* <+\infty$, and both~\eqref{eq:P} and~\eqref{eq:D} have at least one optimal solution. For convenience, we denote the set of optimal solution of~\eqref{eq:P} by $\calX^*$. Since the function $P$ is closed and convex, we know that $\calX^*\ne\emptyset$ is closed and convex. 
%
%

Given $x\in \dom \Psi$, 
let 
\begin{equation}
g:=\nabla f(x) \quad \andd \quad s\in\partial \Psi^*(-g).  \label{eq:def_gs}
\end{equation}
 Then we have 
\begin{equation}
f(x) + f^*(g) + \Psi(s) + \Psi^*(-g) = \ip{g}{x-s}. \label{eq:FW_identity}
\end{equation}
The {\em Wolfe gap} at $x$, denoted by $\wg(x)$, is defined as 
\begin{equation}
\wg(x):= \ip{g}{x-s} + \Psi(x) - \Psi(s).
\end{equation}
(Note that by the definition of $s$,  $\wg(x)$ is independent of the choice of $s\in\partial \Psi^*(-g)$.) 
For any $x,u\in\bbR^n$, define the duality gap at $(x,u)$ as 
\begin{equation}
\gap(x,u) := P(x) + d(u) =  f(x)+\Psi(x) + f^*(u)+\Psi^*(-u). 
\end{equation}
Clearly, the duality gap upper bounds both the primal and dual sub-optimality gaps, i.e., 
\begin{equation}
\gap(x,u) = P(x) - P^* + d(u) - d^*\ge \max\{P(x) - P^*, d(u) - d^*\}. 
\end{equation}
Also, note that by~\eqref{eq:FW_identity}, we have  
\begin{equation}
 \wg(x)=\gap(x,g), \label{eq:wg_duality_gap}
\end{equation}
and consequently, in Algorithm~\ref{algo:CG}, we have
\begin{equation}
 G_k = \wg(x_k)=\gap(x_k,g_k), \quad \forall\,k\ge 0.  \label{eq:wg_duality_gap_k}
\end{equation}

For our later analysis, we need to state the following simple fact. For convenience, let us define 
\begin{equation}
\subopt(x) := P(x) - P^*. 
\end{equation}

\begin{lemma}
Let $x\in \dom\Psi, g = \nabla f(x),$ and $s\in \partial \Psi^*(-g)$. The following three statements are equivalent:
\begin{enumerate}[label=(\roman*)]
\item \label{item:wg=0} $\wg(x) = 0$ 
\item \label{item:subopt=0} $\subopt(x) = 0$
\item \label{item:x_opt} $x$ is an optimal solution of~\eqref{eq:P}
\end{enumerate}
\end{lemma}

\begin{proof}
By~\eqref{eq:wg_duality_gap}, we have~\ref{item:wg=0} $\Rightarrow$~\ref{item:subopt=0}. That~\ref{item:subopt=0} $\Rightarrow$~\ref{item:x_opt} is obvious. Now suppose that~\ref{item:x_opt} holds. 
By first-order optimality condition,  we know that $-g \in \partial \Psi(x)$, and hence 
\begin{equation}
\Psi(s)\ge \Psi(x) - \ipt{g}{s-x} \quad \Longleftrightarrow\quad   \wg(x) = \ip{g}{x-s} + \Psi(x) - \Psi(s)\le 0.
\end{equation}
Since $\wg(x)\ge 0$,  we know that $\wg(x)= 0$. 
\end{proof}

Let  $D_f:\calU\times\calU\to \bbR$ denote the {\em Bregman divergence} of $f$, i.e., 
\begin{equation}
D_f(y,x) = f(y) - f(x) - \ip{\nabla f(x)}{y-x}, \quad \forall\, x,y\in\calU. 
\end{equation}
From~\eqref{eq:FW_identity}, we know that
\begin{equation}\label{eq:basic.gap}
D_f(s,x)  =  f(s) -f(x) - \ip{g}{s-x} = \gap(s,g). 
\end{equation}
Based on $D_f$, we can define $D:\dom\Psi\times \dom\Psi \times[0,1]\rightarrow \bbR$ as follows: 
\begin{align}\label{eq:def.D}
D(x,s,\theta):=D_f(x+\theta(s-x),x) + \Psi(x+\theta(s-x)) - (1-\theta)\Psi(x) - \theta \Psi(s).
\end{align}
The function $D$ plays a pivotal role in the analytic framework developed in~\cite{Pena_23}. 
The importance is demonstrated in the following lemma, which is a restatement of~\cite[Proposition~1]{Pena_23}, wherein it was called the ``fundamental gap reduction identity''.

\begin{lemma}\label{lem:gap.reduction}
Let $x\in \dom\Psi, g = \nabla f(x),$ and $s\in \partial \Psi^*(-g)$.  Then for all $\theta\in[0,1]$, we have
\begin{equation}\label{eq:gap.reduction}
\subopt(x+\theta(s-x)) - \subopt(x) = -\theta\,\gap(x,g) + D(x,s,\theta).
\end{equation}
\end{lemma}



\section{Rate Analysis of $\{G_k^{\rm best}\}_{0\le k \le K}$ Under the Weak Growth Condition}
 
 In this section, we first formally introduce the weak growth condition in~\cite{Pena_23} and make a few remarks. Next, we present  a key recursion lemma that lays the foundation of our analyses. Finally, we analyze the convergence rate of $\{G_k^{\rm best}\}_{0\le k \le K}$ 
 in Algorithm~\ref{algo:CG} under two choices of $\theta_k$: one via exact line search and one via Armijo's rule. 
 
\subsection{The Weak Growth Condition} \label{sec:weak_growth}

The strong and weak growth conditions in~\cite{Pena_23} aim to upper bound the quantity $D(x,s,\theta)$ in~\eqref{eq:def.D} in different ways. Since we focus on the latter condition, we formally introduce it below. 

\begin{definition}[Weak Growth Condition]\label{def:weak.growth.cond}
We say that the 
{\em $(q,r)$-weak growth condition} holds on~\eqref{eq:P}, where $q>1$ and $r\in[0,1)$, if there exists a finite $M \ge 0$ such that for all $x \in \dom\Psi, \, g:=\nabla f(x)$ and $s\in\partial \Psi^*(-g)$, 
\begin{equation}\label{eq:weak.growth.cond}
D(x,s,\theta) \; \subopt(x)^{1-r} \le \frac{M\theta^{q}}{q}\cdot {\gap}(x,g), \quad \forall\, \theta \in [0,1].
\end{equation} 
\end{definition}

Let use make a few remarks about Definition~\ref{def:weak.growth.cond} below. 

\begin{remark}
Note that if $M= 0$, then for any $x\in \dom\Psi\setminus\calX^*,\, g:=\nabla f(x)$, $s\in\partial \Psi^*(-g)$ and $\theta \in [0,1]$, we have $D(x,s,\theta)\le 0$. Thus by taking $\theta=1$ in~\eqref{eq:gap.reduction}, we have  
\begin{align}
\subopt(s) = \subopt(x)  -\,\gap(x,g) + D(x,s,1)\le 0,
\end{align}
and hence $s\in \calX^*$. This indicates that in Algorithm~\ref{algo:CG}, no matter which $x_{0}\in \dom\Psi$ is chosen, we either have $x_0\in \calX^*$ or $s_0\in \calX^*$. 
Thus, to avoid this trivial case,  we shall assume that $M>0$ in the sequel. 
\end{remark}

\begin{remark}
Note that the weak growth condition generalizes and extends the special scenarios studied in the previous works~\cite{Garber_15,Xu_18,Kerd_21}. More importantly, note that we can write 
\begin{equation}
M:= \sup_{x\in \dom \Psi\setminus \calX^*}\;\sup_{\theta\in(0,1]}\;\frac{D(x,s,\theta)\,\subopt(x)^{1-r}}{(\theta^{q}/q)\,{\gap}(x,g)}, 
\end{equation} 
and hence $M$ is affine-invariant and norm-independent. As remarked in~\cite{Pena_23}, the analysis of Algorithm~\ref{algo:CG} based on the weak growth condition, if done properly, is not only affine-invariant and norm-independent, 
but also more general and 
at least as sharp as the norm-dependent results in the previous works~\cite{Garber_15,Xu_18,Kerd_21}.
\end{remark}

\begin{remark}
Note that when $r=1$, we effectively have the $(q,1)$-strong growth condition (cf.~\cite[Eqn.~(10)]{Pena_23}), under which the convergence behavior of Algorithm~\ref{algo:CG} has been thoroughly studied in~\cite{Pena_23}. Since our focus is genuinely on the  weak growth condition, we exclude this case from our discussions. 
\end{remark}


\subsection{A Key Recursion Lemma}

\begin{lemma} \label{lem:rate_G_k}
Let $\alpha, p>0$, $\{G_k\}_{0\le k\le  K}$ and $\{h_k\}_{0\le k\le  K}$ be 
positive sequences (where $K\ge 1$) and $h_{K+1}\ge 0$ such that 
\begin{align}
 h_k - h_{k+1}\ge \alpha G_k h_k^p, \quad G_k\ge h_k,\quad \forall\, 0\le k\le  K. 
 \label{eq:recur1}
\end{align}
Then we have 
\begin{equation}
h_k\le (h_0^{-p} + p\alpha k)^{-1/p},\quad \forall\, 0\le k\le  K.\label{eq:rate_h_k}
\end{equation}
Define $\uG_k:= \min_{i=\ceil{(k+1)/2},\cdots,k} G_i$ for $0\le k\le  K$.  
If $p\in(0,1)$, then 
\begin{equation}
G_0 \le \frac{h_0^{1-p}}{\alpha(1-p)}\quad \andd \quad\min_{i=0,\cdots,k} G_i\le \uG_k \le \left(\frac{2}{\alpha p}\right)^{1/p}\frac{p}{(1-p)k^{1/p}}, \quad\forall\, 1\le  k\, \le K. \label{eq:rate_uG_k}
\end{equation}
\end{lemma}

\begin{remark}\label{rmk:recur}
Several remarks are in order. First, if $p=0$ in~\eqref{eq:recur1}, then $\alpha<1$ and we have 
\begin{equation}
h_k\le (1-\alpha)^k h_0,\quad \forall\, 0\le k\le  K+1, \quad\andd\quad G_k\le \alpha^{-1} (1-\alpha)^k h_0 ,\quad \forall\, 0\le k\le  K. 
\end{equation}
Second, if $p\ge 1$ in~\eqref{eq:recur1}, then $\{\uG_k\}_{0\le k\le  K}$  may not decrease at all. 
 For example, let $h_k = 2^{-(k+1)}$ and $G_k = 1/2$ for all $0\le k\le  K$, then clearly~\eqref{eq:recur1} is satisfied with $\alpha = 1$, but $\uG_k = 1/2$ for all $0\le k\le  K$. Third, we can take $K=+\infty$ in Lemma~\ref{lem:rate_G_k}, and all the results still hold. 
\end{remark}


\begin{proof}[Proof of Lemma~\ref{lem:rate_G_k}]
Note that~\eqref{eq:recur1} implies that
\begin{equation}
h_k - h_{k+1}\ge \alpha h_k^{p+1}, \quad \forall\, 0\le k\le  K. 
\end{equation}
From~\cite[Lemma 4.1]{Borwein_14}, we know that~\eqref{eq:rate_h_k} holds. For completeness, we state the proof below. Fix any $0\le k\le K$. Since both $h_k,h_{k+1}>0$, we have 
\begin{align}
\alpha\le h_k^{-(p+1)}(h_k - h_{k+1})\le  \int_{h_{k+1}}^{h_k}\, x^{-(p+1)}\;\rmd x = \frac{h_{k+1}^{-p} - h_{k}^{-p}}{p}, \quad \forall\, 0\le k\le  K-1. 
\end{align}
Summing over $i=0,\ldots,k$, we have 
\begin{align}
h_{k+1}^{-p}\ge h_{0}^{-p} +p\alpha (k+1) \quad \Longrightarrow\quad  h_{k+1}\le (h_{0}^{-p} +p\alpha (k+1) )^{^{-1/p}}. 
\end{align}
This shows~\eqref{eq:rate_h_k}. 
Now,  let $p\in(0,1)$, and 
we have  
\begin{equation}
G_k \le \frac{1}{\alpha} h_k^{-p}({h_k - h_{k+1}})\le \frac{1}{\alpha} \int_{h_{k+1}}^{h_k}\, x^{-p}\;\rmd x = \frac{h_k^{1-p} - h_{k+1}^{1-p}}{\alpha (1-p)},  \quad \forall\, 0\le k\le  K. \label{eq:G_k_ub}
\end{equation}
This shows the first part of~\eqref{eq:rate_uG_k}. Next, fix any $1\le k\le K$. 
Summing over $i=\ceil{(k+1)/2},\ldots,k$, we have 
\begin{align}
({k}/{2}) \uG_k\lea (k - \ceil{(k-1)/2}+1)  \uG_k\le \sum_{i=\ceil{(k+1)/2}}^{k} G_i \le \frac{h_{\ceil{(k+1)/2}}^{1-p} - h_{k+1}^{1-p}}{\alpha (1-p)}\le \frac{h_{\ceil{(k+1)/2}}^{1-p} }{\alpha (1-p)}, 
\end{align}
where in (a) we use $\ceil{(k+1)/2}\le k/2+1$. Since $\ceil{(k+1)/2}\le \ceil{(K+1)/2}\le K$ for all $K\ge 1$, 
by~\eqref{eq:rate_h_k}, we have 
\begin{align}
\uG_k \le \frac{2(h_0^{-p} + p\alpha \ceil{(k+1)/2})^{-(1-p)/p}}{\alpha (1-p)k}\le \frac{2( p\alpha (k+1)/2)^{-(1-p)/p}}{\alpha (1-p)k}\le   
\left(\frac{2}{\alpha p}\right)^{1/p}\frac{p}{(1-p)k^{1/p}}.\nn
\end{align}
This shows the second part of~\eqref{eq:rate_uG_k}.
\end{proof}

Now, let us turn our attention back to Algorithm~\ref{algo:CG}. As mentioned in Section~\ref{sec:intro}, $K$ denotes the index of the iteration  immediately before Algorithm~\ref{algo:CG} terminates, i.e.,
\begin{equation}
K:= \max\{k\ge 0: G_k > \varepsilon\}. 
\end{equation}
To avoid triviality, we assume that $K\ge 1$. 
For convenience, we also define 
\begin{equation}
h_k:= \subopt(x_k)\quad \andd\quad G_k:= \gap(x_k,g_k), \quad \forall\, k\ge 0.  
\end{equation}
The goal of the rest of this section is to provide convergence rate guarantees for 
$\{G_k^{\rm best}\}_{0\le k \le K}$ 
based on the weak growth condition in Definition~\ref{def:weak.growth.cond}. 

\subsection{Choosing $\theta_k$ via Exact Line Search}

We first consider choosing $\theta_k$ via exact line search on the primal objective function $P$. 
In Algorithm~\ref{algo:CG}, at each iteration $k\ge 0$, we choose $\theta_k\in[0,1]$ as follows: 
\begin{equation}
\theta_k := {\argmin}_{\theta\in[0,1]}\;\; P(x_k+\theta(s_k-x_k)). \label{eq:exact_LS}
\end{equation}

\begin{theorem} \label{thm:rate_exact_LS}
Let the {$(q,r)$-weak growth condition} hold on~\eqref{eq:P} for some $q>1$ and $r\in[0,1)$. 
In Algorithm~\ref{algo:CG}, let $\theta_k$ be chosen according to~\eqref{eq:exact_LS}.
For any $0\le k\le  K$, we have 
\begin{equation}
h_{k+1} - h_k\le \; -\left( 1-\frac{1}{q} \right) \min\bigg\{1,\bigg(\frac{h_k^{1-r}}{M}\bigg)^{\frac{1}{q-1}}\bigg\}\, G_k.  \label{eq:recur1_h_k_G_k}
\end{equation}
Consequently, define 
\begin{equation}
k_0:= \max\big\{0\le k\le K:h_k^{1-r}/M> 1 \big\}. 
\end{equation}
Then we have the following convergence rate guarantee on $\{h_k\}_{0\le k \le K}$: 
\begin{align}
&h_k \le q^{-k} h_0, \quad \forall\, 0\le k\le k_0+1,\label{eq:rate_h_k_1st}\quad\andd\quad\\
 &h_k\le \left(h_{k_0+1}^{-\frac{1-r}{q-1}} + \frac{ 1-r }{q M^{\frac{1}{q-1}}} (k-(k_0+1))\right)^{-\frac{q-1}{1-r}},\quad \forall\, k_0+2\le  k\le K.\label{eq:rate_h_k_2nd_thm}
\end{align}
In addition,  we have the following convergence rate guarantee on $\{G_k\}_{0\le k \le K}$:
\begin{align}
&G_k\le \left( 1-{1}/{q} \right)^{-1}q^{-k}h_0, \quad \forall\, 0\le k\le k_0,
\end{align}
and if $r+q>2$, then we have 
\begin{align}
&G_{k_0+1} \le q^{-(k_0+1)\frac{r+q-2}{q-1}+1}\;\frac{M^{\frac{1}{q-1}} }{q+r-2}h_0^{\frac{r+q-2}{q-1}},\quad \andd \quad\\
 & \min_{i=k_0+1,\cdots,k} G_i\le \left(\frac{2q M^{\frac{1}{q-1}}}{1-r}\right)^{\frac{q-1}{1-r}}\frac{1-r}{(q+r-2)(k-(k_0+1))^{\frac{q-1}{1-r}}}, \quad\forall\, k_0+2\le  k\, \le K. \label{eq:rate_uG_k_2nd}
\end{align}
\end{theorem}

\begin{proof}
Fix any $0\le k\le  K$, so that both $h_k,G_k>0$. 
By~\eqref{eq:gap.reduction} and~\eqref{eq:exact_LS}, we have 
\begin{equation}\label{eq:gap.reduction_k}
\subopt(x_k+\theta(s_k-x_k)) - \subopt(x_k) = -\theta\,\gap(x_k,g_k) + D(x_k,s_k,\theta).
\end{equation}
Therefore, we have 
\begin{align}
h_{k+1} - h_k=\;&\subopt(x_{k+1}) - \subopt(x_k)\\
=\;&{\min}_{\theta\in[0,1]}\;\subopt(x_k+\theta(s_k-x_k)) - \subopt(x_k)\\
 =\; &{\min}_{\theta\in[0,1]}\; -\theta\,\gap(x_k,g_k) + D(x_k,s_k,\theta)\\
 \lea \; &{\min}_{\theta\in[0,1]}\; -\theta\,G_k +   \frac{M\theta^{q}}{q}\frac{G_k}{h_k^{1-r}}\\
  =\; &\theta_*G_k\left( -1 +   \frac{M\theta_*^{q-1}}{qh_k^{1-r}}\right), \quad \where \;\;\theta_* := \min\bigg\{1,\bigg(\frac{h_k^{1-r}}{M}\bigg)^{\frac{1}{q-1}}\bigg\}\\
  \le \; &-\left( 1-\frac{1}{q} \right) \min\bigg\{1,\bigg(\frac{h_k^{1-r}}{M}\bigg)^{\frac{1}{q-1}}\bigg\}\, G_k \lb 0,
\end{align}
where in (a) follows from~\eqref{eq:weak.growth.cond} and in (b) follows from $q>1$. This shows~\eqref{eq:recur1_h_k_G_k}. Also, we know that $\{h_k\}_{0\le k\le  K}$ is strictly decreasing. We then have the following: 
\begin{enumerate}[label=\roman*)]
\item For all $0\le k\le k_0$, we have $h_k^{1-r}/M> 1$, and hence 
\begin{align}
&h_{k+1} - h_k\le \; -\left( 1-{1}/{q} \right)  G_k\le -\left( 1-{1}/{q} \right)  h_k\\
 \quad\Longrightarrow \quad & h_{k+1}\le h_k/q \quad\andd\quad  G_k\le \left( 1-{1}/{q} \right)^{-1}(h_k-h_{k+1})\le \left( 1-{1}/{q} \right)^{-1}h_k. 
\end{align}
Therefore, we have 
\begin{equation}
h_k \le q^{-k} h_0, \quad \forall\, 0\le k\le k_0+1,\quad\andd\quad G_k\le \left( 1-{1}/{q} \right)^{-1}q^{-k}h_0, \quad \forall\, 0\le k\le k_0. 
\end{equation}
\item For all $k_0+1\le k\le K$, we have $h_k^{1-r}/M\le  1$,  and hence
\begin{align}
h_{k+1} - h_k\ge  \alpha h_k^{p}\, G_k,\quad \where\quad \alpha:= \frac{ 1-{q}^{-1} }{M^{\frac{1}{q-1}}},\;\; p:=\frac{1-r}{q-1}. \label{eq:recur_h_k_G_k}
\end{align}
Therefore, the positive sequences $\{G_k\}_{k_0+1\le k\le  K}$ and $\{h_k\}_{k_0+1\le k\le  K}$, together with $h_{K+1}\ge 0$, satisfy~\eqref{eq:recur1} for $k_0+1\le k\le  K$. 
Thus by~\eqref{eq:rate_h_k} in Lemma~\ref{lem:rate_G_k}, we have
\begin{equation}
h_k\le \Big(h_{k_0+1}^{-p} + p\alpha (k-(k_0+1))\Big)^{-1/p},\quad \forall\, k_0+1\le k\le K. \label{eq:rate_h_k_2nd} 
\end{equation}
Since $r<1$ and $r+q>2$, we know that $ p\in(0,1)$, and by~\eqref{eq:rate_uG_k} in Lemma~\ref{lem:rate_G_k}, we have 
\begin{align}
&G_{k_0+1} \le \frac{h_{k_0+1}^{1-p}}{\alpha(1-p)}\le \frac{q^{-(k_0+1)(1-p)}h_0^{1-p}}{\alpha(1-p)}\quad \andd \quad\\
 & \min_{i=k_0+1,\cdots,k} G_i\le \left(\frac{2}{\alpha p}\right)^{1/p}\frac{p}{(1-p)(k-(k_0+1))^{1/p}}, \quad\forall\, k_0+2\le  k\, \le K. \label{eq:rate_G_k_2nd} 
\end{align}
\end{enumerate}
Substituting the values of $\alpha$ and $p$ into~\eqref{eq:rate_h_k_2nd} to~\eqref{eq:rate_G_k_2nd}, we then finish the proof. 
\end{proof}

\begin{remark}
Note that from~\eqref{eq:rate_h_k_1st}, we can have an explicit estimate of $k_0$, that is 
\begin{equation}
k_0 \le \min\left\{\left\lceil\frac{q}{q-1}\ln\left(\frac{h_0}{M^{1/(1-r)}}\right)\right\rceil-1,\;K\right\}.
\end{equation}
\end{remark}

\begin{remark}
Note that by the definition of $G_k^{\rm best}$ in~\eqref{eq:def_G_k_best} and also~\eqref{eq:wg_duality_gap_k},  we have 
\begin{equation}
\begin{split}
G_k^{\rm best}&:={\min}_{i=0,\ldots,k}\,G_i = {\min}_{i=0,\ldots,k}\,P(x_i) + d(g_i)\\
&\qquad \ge {\min}_{i=0,\ldots,k}\,P(x_i) + {\min}_{i=0,\ldots,k}\,d(g_i) = P(x_k) + {\min}_{i=0,\ldots,k}\,d(g_i):=\tilG_k, \label{eq:G_best_lb}
\end{split}
\end{equation}
where the last equality follows from the monotonicity of $\{h_k\}_{k\ge 0}$ (cf.~\eqref{eq:recur1_h_k_G_k}). 
Note that the convergence rate of $\{\tilG_k\}_{0\le k\le K}$ was analyzed under the {\em strong growth condition} in~\cite{Pena_23}. From~\eqref{eq:G_best_lb}, it is clear that our convergence rate guarantees for $\{G_k^{\rm best}\}_{0\le k\le K}$, which is derived under the weak growth condition, also holds for $\{\tilG_k\}_{0\le k\le K}$. 
\end{remark}

\begin{remark}
Note that the $O(k^{-\frac{q-1}{1-r}})$ convergence rate on $\{h_k\}_{0\le k \le K}$ has already been derived in~\cite[Theorem~3]{Pena_23}. 
Our convergence rate guarantee on $\{G_k\}_{0\le k \le K}$ complements this result, by showing that 
the sequence of ``best'' duality gaps $\{G_k^{\rm best}\}_{0\le k \le K}$ 
has the same convergence rate  of $O(k^{-\frac{q-1}{1-r}})$, but only under the regime where $r+q>2$. Indeed, the requirement that $r+q>2$ stems from the requirement that $p\in(0,1)$ in Lemma~\ref{lem:rate_G_k}, which provides the (worst-case) convergence rate of $\{G_k^{\rm best}\}_{0\le k \le K}$ based on the recursion in~\eqref{eq:recur_h_k_G_k}. As discussed in Remark~\ref{rmk:recur}, this requirement cannot be dropped unless we have additional information about the relationship between 
$\{h_k\}_{0\le k \le K}$ and $\{G_k\}_{0\le k \le K}$. 
In fact, it is interesting to investigate the behavior of $\{G_k^{\rm best}\}_{0\le k \le K}$ (or some other sequences of duality gaps) in the regime where $r+q\le 2$, 
and we leave this to future work. 
\end{remark}

\subsection{Choosing $\theta_k$ via Armijo's Rule} 

We next consider choosing $\theta_k$ via Armijo's Rule on the primal objective function $P$.
In Algorithm~\ref{algo:CG}, fix an iteration $0\le k\le K$, so that $G_k>\varepsilon$ and $x_k\ne s_k$. 
Define the function $\varphi:[0,1]\rightarrow \bbR$ as 
\begin{equation}
\varphi(\theta):= P(x_k+\theta(s_k-x_k)), 
\quad\forall\,\theta\in [0,1]. \label{eq:def_phi} 
\end{equation}
Clearly, $\varphi(0) = P(x_k)$. 
We define the right derivative of $\varphi$ at zero as 
\begin{equation}
\varphi'_+(0) := \lim_{\theta\downarrow0}\, \frac{\varphi(\theta) - \varphi(0)}{\theta} = \inf_{\theta\in(0,1]}\, \frac{\varphi(\theta) - \varphi(0)}{\theta}, \label{eq:def_phi'} 
\end{equation}
where the second equality follows from the convexity of $\varphi$.
By the definition of $\varphi$ in~\eqref{eq:def_phi}, we have 
\begin{align}
\varphi'_+(0) &= \lim_{\theta\downarrow0}\, \frac{f(x_k+\theta(s_k-x_k)) - f(x_k)}{\theta} + \lim_{\theta\downarrow0}\, \frac{\Psi(x_k+\theta(s_k-x_k)) - \Psi(x_k)}{\theta}\\
&\le \ipt{g_k}{s_k-x_k} + \Psi(s_k) - \Psi(x_k)\\
& = -\wg(x_k) = -\gap(x_k,g_k) = -G_k<0, 
\end{align}
where the first inequality follows from the convexity of $\Psi$. As such, we know that for any $\beta\in(0,1)$, 
\begin{align}
\varphi'_+(0)\le -G_k < -\beta G_k<0,  
\end{align}
and by~\eqref{eq:def_phi'}, we know that there exist $\tilde\theta\in(0,1]$ such that 
\begin{align}
\frac{\varphi(\tilde\theta) - \varphi(0)}{\tilde\theta} < -\beta G_k \quad \Longleftrightarrow\quad  \varphi(\tilde\theta) -  \varphi(0)<  - \beta \tilde\theta G_k . 
\end{align}
This, together with the closedness an convexity of $\varphi$, suggests that 
\begin{equation}
\{\theta\in[0,1]:\varphi(\theta) -  \varphi(0)\le   - \beta \theta G_k \}=[0,\bartheta]\quad \mbox{for some }\;0<\bartheta\le 1. 
\end{equation}
Thus 
 it is natural for us to employ Armijo's rule 
 to choose $\theta_k\in[0,1]$ such that 
\begin{equation}
P(x_{k+1}) - P(x_k)=\varphi(\theta_k) - \varphi(0) \le -\beta \theta_kG_k \quad \andd \quad c\, \bartheta< \theta_k \le \bartheta\quad \mbox{for some }\;c\in(0, 1).  \label{eq:theta_k_crit}
\end{equation} 
Based on this criterion, 
we have the following result. 

\begin{theorem} \label{thm:rate_Armijo}
Let the {$(q,r)$-weak growth condition} hold on~\eqref{eq:P} for some $q>1$ and $r\in[0,1)$. 
In Algorithm~\ref{algo:CG}, let $\theta_k$ be chosen via Armijo's rule so that~\eqref{eq:theta_k_crit} holds. 
For any $0\le k\le  K$, we have 
\begin{equation}
h_{k+1} - h_k \le -\beta c\min\bigg\{1,\bigg(\frac{q(1-\beta)h_k^{1-r}}{M}\bigg)^{\frac{1}{q-1}}\bigg\} G_k.  \label{eq:recur2_h_k_G_k}
\end{equation}
Consequently, define 
\begin{equation}
k_0:= \max\big\{0\le k\le K:q(1-\beta)h_k^{1-r}/M> 1 \big\}. 
\end{equation}
Then we have the following convergence rate guarantee on $\{h_k\}_{0\le k \le K}$: 
\begin{align}
&h_k \le (1-\beta c)^{k} h_0, \quad \forall\, 0\le k\le k_0+1,\quad\andd\quad\\
 &h_k\le \left(h_{k_0+1}^{-\frac{1-r}{q-1}} + \frac{ \beta c (1-r) }{q-1 } \left(\frac{q(1-\beta)}{M}\right)^{\frac{1}{q-1}} (k-(k_0+1))\right)^{-\frac{q-1}{1-r}},\quad \forall\, k_0+2\le  k\le K.
\end{align}
In addition,  we have the following convergence rate guarantee on $\{G_k\}_{0\le k \le K}$:
\begin{align}
&G_k\le ( \beta c)^{-1}(1-\beta c)^{k}h_0, \quad \forall\, 0\le k\le k_0,
\end{align}
and if $r+q>2$, then we have 
\begin{align*}
&G_{k_0+1} \le (1-\beta c)^{(k_0+1)\frac{r+q-2}{q-1}}\;\left(\frac{M}{q(1-\beta)}\right)^{\frac{1}{q-1}}\frac{ q-1}{(q+r-2)\beta c} h_0^{\frac{r+q-2}{q-1}},\quad \andd \quad\\
 & \min_{i=k_0+1,\cdots,k} G_i\le \left( \frac{M}{q(1-\beta)}\right)^{\frac{1}{1-r}} \left(\frac{2(q-1) }{\beta c (1-r)}\right)^{\frac{q-1}{1-r}}\frac{1-r}{(q+r-2)(k-(k_0+1))^{\frac{q-1}{1-r}}}, \quad\forall\, k_0+2\le  k\, \le K. 
\end{align*}
\end{theorem}

\begin{proof}
Fix any $0\le k\le  K$, so that both $h_k,G_k>0$. 
Note that by~\eqref{eq:theta_k_crit}, we have 
\begin{equation}
h_{k+1} - h_k =  p(x_{k+1}) - p(x_k) \le -\beta \theta_k G_k \le -\beta c\bar \theta G_k.\label{eq:h_ub}
\end{equation}
From~\eqref{eq:gap.reduction_k} and~\eqref{eq:weak.growth.cond}, we have 
\begin{equation}
\varphi(\theta) -  \varphi(0) = \subopt(x_k+\theta(s_k-x_k)) - \subopt(x_k) \le -\theta\,G_k +   \frac{M\theta^{q}}{q}\frac{G_k}{h_k^{1-r}}. \label{eq:varphi_ub}
\end{equation}
Therefore, 
\begin{align}
\bar\theta &= \max\{\theta\in[0,1]:\varphi(\theta) -  \varphi(0)\le   - \beta \theta G_k \}\\
&\ge \max\left\{\theta\in[0,1]:-\theta\,G_k +   \frac{M\theta^{q}}{q}\frac{G_k}{h_k^{1-r}}\le   - \beta \theta G_k \right\} = \min\bigg\{1,\bigg(\frac{q(1-\beta)h_k^{1-r}}{M}\bigg)^{\frac{1}{q-1}}\bigg\}.\label{eq:bar_theta_lb}
\end{align}
By combining~\eqref{eq:h_ub} and~\eqref{eq:bar_theta_lb}, we have~\eqref{eq:recur2_h_k_G_k}. 
Based on~\eqref{eq:recur2_h_k_G_k}, we can leverage Lemma~\ref{lem:rate_G_k} and use the same reasoning as in the proof of Theorem~\ref{thm:rate_exact_LS} to finish the proof.
\end{proof}

\noindent
From Theorem~\ref{thm:rate_Armijo}, we see that even if $\theta_k$ is chosen via Armijo's rule (instead of  exact line search), the sequence of ``best'' duality gaps $\{G_k^{\rm best}\}_{0\le k \le K}$ still has the same convergence rate  of $O(k^{-\frac{q-1}{1-r}})$. 

%

\section{An Important Scenario Satisfying the Weak Growth Condition}

In the last part of this paper, we introduce one important scenario where the weak growth condition holds. To that end, we need to introduce some definitions first. 

\begin{definition}[Uniformly Convex Set]\label{def:uniformly_cvx_set}
Let $\normt{\cdot}$ be a norm on $\bbR^n$. 
A closed set $\calC\subseteq\bbR^n$ is {\em $(\mu,p)$-uniformly convex} w.r.t.\ $\normt{\cdot}$, 
where $\mu>0$ and $p\ge 2$, if for all $x,y\in \calC$ and  $\theta \in [0,1]$, we have 
\begin{equation}\label{eq.unif.convex.set}
(1-\theta)x+\theta y + \frac{\mu}{p}\theta(1-\theta)\|y-x\|^p z \in \calC,\quad \forall\,z\in\bbR^n\;\st\;  \|z\|\le 1. 
\end{equation}
\end{definition}

\begin{remark}
Note that in Definition~\ref{def:uniformly_cvx_set}, we require $p\ge 2$, as opposed to some literature (e.g.,~\cite[Definition~1.1]{Kerd_21}) that only requires $p> 0$. This is because one can show that if $n\ge 2$ and $\calC\ne \bbR^n$ contains at least two points, then we must have $p\ge 2$. In addition, note that for any $p>1$, we either have $\calC= \bbR^n$ or $\calC$ is compact. Therefore, if $p\ge 2$ and $s\ne 0$, then $\argmin_{x\in\calC}\, \ipt{s}{x}\ne \emptyset$ if and only if $\calC$ is compact. 
Let $x^*\in \argmin_{x\in\calC}\, \ipt{s}{x}$, and we know that (see e.g.,~\cite[Lemma~2.1]{Kerd_21})
\begin{equation}
\ipt{s}{x-x^*}\ge \frac{\mu}{2p}\normt{s}_*\|x-x^*\|^p, \quad \forall\, x\in \calC. 
\end{equation}
In particular, if $s\ne 0$ and $x^*$ exists, then it 
must be unique. 
\end{remark}

\begin{definition}[Uniformly Smooth Function]
Let $\normt{\cdot}$ be a norm on $\bbR^n$. 
We call the function $f$ {\em $(L,q)$-uniformly smooth} on $\calU$ w.r.t.\ $\|\cdot\|$, where  $q\in(1,2]$ and $L>0$, if 
\begin{equation}\label{eq.unif.smooth}
f(x+\theta(y-x)) \ge (1-\theta)f(x) + \theta f(y) - \frac{L}{q} \theta(1-\theta) \|y-x\|^q, \quad \forall\, x,y\in \calU,\;\; \forall\,\theta \in [0,1]. 
\end{equation}
Note when $q=2$, we typically call $f$ {\em $L$-smooth} on $\calU$. 
 \end{definition}

Before proceeding, let us introduce some notations. Let $\calC\subseteq \bbR^n$ be a nonempty closed convex set  and $\normt{\cdot}$ be a norm on $\bbR^n$. We let  $\iota_\calC$ denote the indicator function of $\calC$, namely $\iota_\calC(x) = 0$ for $x\in\calC$ and $\iota_\calC(x) = +\infty$ otherwise. 
Also, let $\diam_{\normt{\cdot}}(\calC)$ denote the diameter of $\calC$ w.r.t.\ $\normt{\cdot}$, namely $$\diam_{\normt{\cdot}}(\calC):={\sup}_{x,y\in \calC}\, \|x-y\|.$$
Given $x\in\bbR^n$, define its distance to $\calX^*$ w.r.t.\ $\normt{\cdot}$ as $$\dist_{\normt{\cdot}}(x,\calX^*):=  {\inf}_{y\in \calX^*} \,\|x-y\|.$$

\begin{definition}[H\"olderian Error Bound]
Let $\calC\subseteq \bbR^n$ be a nonempty closed convex set  and $\normt{\cdot}$ be a norm on $\bbR^n$. Let $\Psi=\iota_\calC$ in~\eqref{eq:P}, which then becomes $P^*=\min_{x\in\calC} f(x)$. 
We say that $f$ 
satisfies the $(\rho,\gamma)$-H\"olderian error bound on $\calC$ w.r.t.\ $\normt{\cdot}$, where $\rho>0$  and $\gamma \in [0,1]$, if 
\begin{equation}\label{eq.heb}
\dist_{\normt{\cdot}}(x,\calX^*)\le \rho  (f(x) - P^*)^\gamma,\quad\forall\, x\in\calC.  
\end{equation} 
 \end{definition}
%

The following lemma is a combination of Proposition~2 and Proposition~5 in~\cite{Pena_23}. 
 
\begin{lemma}\label{lem:weak_growth}
Let $\calC\subseteq \bbR^n$ be a nonempty convex compact set and $\Psi = \iota_\calC$. Also, let  $f$ be $(L,q)$-uniformly smooth on $\calC$ w.r.t.\ some norm $\normt{\cdot}$ for some $L>0$ and $q \in (1,2]$. Then 
\begin{equation}\label{eq:r=0}
D(x,s,\theta) \le \frac{M\theta^{q}}{q},\quad\forall\, \theta \in [0,1], \qquad \where\;\; M=L\,\diam_{\normt{\cdot}}(\calC)^q.  
\end{equation}
Additionally, if $f$ satisfies the $(\rho,\gamma)$-H\"olderian error bound~\eqref{eq.heb} w.r.t.\ $\normt{\cdot}$ for some $\rho>0$ and $\gamma \in [0,1]$, 
and $\calC$ is $(\mu,p)$-uniformly convex w.r.t.\ $\normt{\cdot}$ for some $\mu>0$ and $p\ge 2$, then the $(q,r)$-weak growth condition in~\eqref{eq:weak.growth.cond} holds for $r = \gamma q/p$ and $M =  L\left(2p\rho/\mu\right)^{q/p}.$ 
\end{lemma}  

\begin{remark}
Note that~\eqref{eq:r=0} is referred to as the $(q,0)$-strong growth condition in~\cite{Pena_23}, and it holds 
regardless of the $(\mu,p)$-uniform convexity of $\calC$ and the $(\rho,\gamma)$-H\"olderian error bound in~\eqref{eq.heb}.  In fact, it was shown in~\cite[Theorems~1 and 2]{Pena_23} that under this condition and proper choices of $\theta_k$ in Algorithm~\ref{algo:CG}, we have a $O(k^{-(q-1)})$ convergence rate for both the sequence of sub-optimality gaps and certain  sequence of duality gaps. In addition, in~\cite[Theorems~3 and 4]{Pena_23}, it was shown that under the $(q,r)$-weak growth condition in~\eqref{eq:weak.growth.cond} and proper choices of $\theta_k$ in Algorithm~\ref{algo:CG}, the sequence of sub-optimality gaps converge at a rate of $O(k^{-\frac{q-1}{1-\gamma q/p}})$, but no convergence rate was obtained for any sequence of duality gaps in~\cite{Pena_23} (as well as in the previous works~\cite{Garber_15,Xu_18,Kerd_21}). 
Observe that  the weak growth condition is valuable mainly in the case where $\gamma>0$, since it  leads to an improvement of the convergence rate of the sub-optimality gaps. 
Indeed, when $\gamma=0$, we can directly leverage~\eqref{eq:r=0} to obtain the $O(k^{-(q-1)})$ convergence rate for both the sequence of sub-optimality gaps and certain  sequence of duality gaps, as mentioned above.  That said, note that for certain problem instances, the value of $M$ in the $(q,0)$-weak growth condition 
can be significantly smaller than its value in the $(q,0)$-strong growth condition. In this situation, it may be more advantageous to use the weak growth condition, instead of its strong counterpart, to obtain a constant-tighter   $O(k^{-(q-1)})$ convergence rate of the sub-optimality gaps. 
\end{remark}

Based on Lemma~\ref{lem:weak_growth} and Theorems~\ref{thm:rate_exact_LS} and~\ref{thm:rate_Armijo}, we have the following corollary.

\begin{corollary} \label{cor:exact_LS}
Let $\normt{\cdot}$ be a norm on $\bbR^n$, $\calC\subseteq \bbR^n$ be a nonempty convex compact set,  and $\Psi = \iota_\calC$. Also, let  $f$ be $(L,q)$-uniformly smooth on $\calC$ w.r.t.\ $\normt{\cdot}$ 
for some $L>0$ and $q=2$, and satisfy the $(\rho,\gamma)$-H\"olderian error bound~\eqref{eq.heb} w.r.t.\ $\normt{\cdot}$ for some $\rho>0$ and $\gamma \in [0,1]$. 
In addition, let $\calC$ be $(\mu,p)$-uniformly convex w.r.t.\ $\normt{\cdot}$ for some $\mu>0$ and $p\ge 2$.  If $\theta_k$ is chosen according to~\eqref{eq:exact_LS} or~\eqref{eq:theta_k_crit} in Algorithm~\ref{algo:CG}, then $\{G_k^{\rm best}\}_{0\le k \le K}$ has 
a convergence rate of $O(k^{-\frac{1}{1-2\gamma /p}})$. 
\end{corollary}

\begin{proof}
We first consider the case where $\gamma=0$. Since $\calC\subseteq \bbR^n$ is convex compact and $f$ is $L$-smooth on $\calC$, the $O(k^{-1})$ convergence rate of $\{G_k^{\rm best}\}_{0\le k \le K}$ is rather standard in the literature (see e.g.,~\cite[Proposition 4.3]{Bach_15}). We next consider the case where $\gamma>0$.
Since $q=2$, from Lemma~\ref{lem:weak_growth}, we know that the $(2,r)$-weak growth condition in~\eqref{eq:weak.growth.cond} holds for $r = 2\gamma /p>0$. Since $r+q>2$, we invoke    Theorems~\ref{thm:rate_exact_LS} and~\ref{thm:rate_Armijo} to complete the proof. 
\end{proof}

\bibliographystyle{abbrv}
\bibliography{math_opt}

\begin{thebibliography}{10}

\bibitem{Bach_15}
F.~Bach.
\newblock Duality between subgradient and conditional gradient methods.
\newblock {\em SIAM J. Optim.}, 25(1):115--129, 2015.

\bibitem{Borwein_14}
J.~M. Borwein, G.~Li, and L.~Yao.
\newblock Analysis of the convergence rate for the cyclic projection algorithm
  applied to basic semialgebraic convex sets.
\newblock {\em SIAM J. Optim.}, 24(1):498--527, 2014.

\bibitem{Frank_56}
M.~Frank and P.~Wolfe.
\newblock An algorithm for quadratic programming.
\newblock {\em Nav. Res. Logist. Q.}, 3(1‐2):95--110, 1956.

\bibitem{Freund_16}
R.~M. Freund and P.~Grigas.
\newblock New analysis and results for the frank–wolfe method.
\newblock {\em Math. Program.}, 155:199–--230, 2016.

\bibitem{Garber_15}
D.~Garber and E.~Hazan.
\newblock Faster rates for the frank-wolfe method over strongly-convex sets.
\newblock In {\em Proc. ICML}, pages 541--549, 2015.

\bibitem{Ghad_19}
S.~Ghadimi.
\newblock Conditional gradient type methods for composite nonlinear and
  stochastic optimization.
\newblock {\em Math. Program.}, 173:431--–464, 2019.

\bibitem{Jaggi_13}
M.~Jaggi.
\newblock Revisiting {Frank-Wolfe}: Projection-free sparse convex optimization.
\newblock In {\em Proc.\ ICML}, pages 427--435, 2013.

\bibitem{Kerd_21}
T.~Kerdreux, A.~d'Aspremont, and S.~Pokutta.
\newblock Projection-free optimization on uniformly convex sets.
\newblock In {\em Proc. AISTATS}, pages 19--27, 2021.

\bibitem{Levitin_66}
E.~S. Levitin and B.~T. Polyak.
\newblock Constrained minimization methods.
\newblock {\em U.S.S.R. Comput. Math. Math. Phys.}, 6(5):1--50, 1966.

\bibitem{Nest_18}
Y.~Nesterov.
\newblock Complexity bounds for primal-dual methods minimizing the model of
  objective function.
\newblock {\em Math. Program.}, 171:311--–330, 2018.

\bibitem{Pena_23}
J.~F. Pe{\~n}a.
\newblock Affine invariant convergence rates of the conditional gradient
  method.
\newblock {\em SIAM J. Optim.}, 33(4):2654--2674, 2023.

\bibitem{Rock_70}
R.~T. Rockafellar.
\newblock {\em Convex analysis}.
\newblock Princeton University Press, 1970.

\bibitem{Xu_18}
Y.~Xu and T.~Yang.
\newblock Frank-wolfe method is automatically adaptive to error bound
  condition.
\newblock arXiv:1810.04765, 2018.

\end{thebibliography}

\end{document}